\documentclass[reqno,a4paper,12pt]{amsart}

\usepackage{amsmath,amscd,amsfonts,amssymb}
\usepackage{mathrsfs,dsfont}

\usepackage{cases}

\numberwithin{equation}{section}
\numberwithin{figure}{section}

\ifdefined\MIHALIS
\else
\fi

\renewcommand{\labelenumi}{\theenumi}

\ifdefined\MIHALIS  
\usepackage{txfonts,pxfonts,tikz} 
\usepackage{tgschola}
\usepackage[T1]{fontenc}
\usepackage[allcolors=blue,allbordercolors=blue,pdfborderstyle={/S/U/W 1}]{hyperref}
\usepackage[ocgcolorlinks]{ocgx2}
\usepackage{bbm}

\ifdefined\SMART 
\usepackage[paperwidth=12cm,paperheight=24cm,left=5mm,right=5mm,top=10mm,bottom=5mm]{geometry}
\else
\calclayout
\fi

\fi 

\newcommand{\RR}{{\mathbb R}}

\newcommand\R{\mathbb{R}}

\newcommand\Z{\mathbb{Z}}

\newcommand\T{\mathbb{T}}
\newcommand\al{\alpha}

\newcommand\gam{\gamma}
\newcommand\Gam{\Gamma}

\newcommand\Del{\Delta}
\newcommand\lam{\lambda}
\newcommand\Lam{\Lambda}

\newcommand\sig{\sigma}
\newcommand\Om{\Omega}

\newcommand\1{\mathds{1}}
\newcommand\eps{\varepsilon}
\newcommand\pphi{\varphi}

\renewcommand\le{\leqslant}
\renewcommand\ge{\geqslant}
\renewcommand\leq{\leqslant}

\newcommand\sbt{\subset}

\newcommand{\mes}{\operatorname{mes}}

\newcommand{\Vol}{\operatorname{Vol}}

\newcommand{\alta}{A'}
\newcommand{\altb}{B'}
\newcommand{\bde}{\stackrel{\text{bd}}{\sim}}

\theoremstyle{plain}
\newtheorem{thm}{Theorem}[section]
\newtheorem{theorem}[thm]{Theorem}

\newtheorem{lemma}[thm]{Lemma}

\newtheorem{proposition}[thm]{Proposition}

\newtheorem*{claim*}{Claim}

\newcommand{\thmref}[1]{Theorem~\ref{#1}}
\newcommand{\secref}[1]{Section~\ref{#1}}

\newcommand{\lemref}[1]{Lemma~\ref{#1}}

\newcommand{\propref}[1]{Proposition~\ref{#1}}

\theoremstyle{definition}
\newtheorem{definition}[thm]{Definition}
\newtheorem*{definition*}{Definition}
\newtheorem*{remarks*}{Remarks}
\newtheorem*{remark*}{Remark}

\newenvironment{enumerate-roman}
{\begin{enumerate}
\addtolength{\itemsep}{5pt}
\renewcommand\theenumi{(\roman{enumi})}}
{\end{enumerate}}

\newenvironment{enumerate-alph}
{\begin{enumerate}
\addtolength{\itemsep}{5pt}
\renewcommand{\theenumi}{(\alph{enumi})}}
{\end{enumerate}}

\newenvironment{enumerate-num}
{\begin{enumerate}
\addtolength{\itemsep}{5pt}
\renewcommand{\theenumi}{(\roman{enumi})}}
{\end{enumerate}}

\newenvironment{enumerate-text}
{\begin{enumerate}
\addtolength{\itemsep}{5pt}
\renewcommand{\theenumi}{\arabic{enumi}.}}
{\end{enumerate}}

\begin{document}

\ifdefined\SMART
\title
[Bounded remainder sets etc]
{Bounded remainder sets, bounded distance equivalent cut-and-project sets, and equidecomposability}
\else
\title
[Bounded remainder sets, cut-and-project sets and equidecomposability]
{Bounded remainder sets, bounded distance equivalent cut-and-project sets, and equidecomposability}
\fi

\author[M. Etkind]{Mark Mordechai Etkind}
\address{Department of Mathematics, Bar-Ilan University, Ramat-Gan 5290002, Israel}
\email{mark.etkind@mail.huji.ac.il}

\author[S. Grepstad]{Sigrid Grepstad} 
\address{Department of Mathematical Sciences, Norwegian University of Science and Technology (NTNU), NO-7491 Trondheim, Norway} 
\email{sigrid.grepstad@ntnu.no}

\author[M. Kolountzakis]{Mihail N. Kolountzakis}
\address{Department of Mathematics and Applied Mathematics, University of Crete, Voutes Campus, 70013 Heraklion, Greece and Institute of Computer Science, Foundation of Research and Technology Hellas, N. Plastira 100, Vassilika Vouton, 700 13, Heraklion, Greece}
\email{kolount@gmail.com}

\author[N. Lev]{Nir Lev}
\address{Department of Mathematics, Bar-Ilan University, Ramat-Gan 5290002, Israel}
\email{levnir@math.biu.ac.il}

\date{June 12, 2026}


\thanks{Research supported by ISF Grant 854/25 and Grant 334466 of the Research Council of Norway.}

\begin{abstract}
We use the measurable Hall's theorem due to Cie\'sla and Sabok to prove that (i) if two measurable sets $A,B \subset \mathbb{R}^d$ of the same measure are bounded remainder sets with respect to a totally irrational $d$-dimensional vector $\alpha$, then $A, B$ are equi\-decomposable with measurable pieces using translations from $\mathbb{Z} \alpha + \mathbb{Z}^d$; and (ii) given a lattice $\Gamma \subset \mathbb{R}^m \times \mathbb{R}^n$ with projections $p_1$ and $p_2$ onto $\mathbb{R}^m$ and $\mathbb{R}^n$ respectively, if two cut-and-project sets in $\mathbb{R}^m$ obtained from Riemann measurable windows $W, W' \subset \mathbb{R}^n$ are bounded distance equivalent, then $W, W'$ are equidecomposable with measurable pieces using translations from $p_2(\Gamma)$. We also prove by a different method that for one-dimensional cut-and-project sets, if the windows $W, W' \subset \mathbb{R}^n$ are polytopes then the pieces can also be chosen to be polytopes; however this result fails in dimensions two and higher.
\end{abstract}

\maketitle

\ifdefined\MIHALIS
\tableofcontents
\fi


\section{Introduction}

\subsection{}
Let $X$ be a set endowed with a group of transformations $G$. Two subsets $A, B \sbt X$ are called \emph{$G$-equidecomposable} if they can be partitioned into the same \textit{finite} number of pieces $A = \bigcup_{i=1}^n A_i$, $B = \bigcup_{i=1}^n B_i$, which can be pairwise matched via elements of $G$, i.e.\ $B_i = g_i \cdot A_i$ for some $g_i \in G$, $i=1, 2, \ldots, n$, where $g \, \cdot $ denotes the group action.

A famous example of equidecomposability is the so-called Tarski circle squaring problem, which was posed by Tarski (1925) \cite{TW16}: is a square of area $1$ equidecomposable to a disk of area $1$ via plane isometries? This was answered in the affirmative by Laczkovich \cite{Lac90}: the square of unit area can be partitioned into a finite number of pieces which can then be translated to form a partition of a disk of unit area (thus the group of transformations of the plane used is not the whole group of isometries but merely the group of translations). Moreover, it was proved by Grabowski, M\'ath\'e and Pikhurko \cite{GMP17} that the pieces in this result can be chosen Lebesgue measurable.

In the present paper we consider the case where $G$ is a finitely generated group of translations of $\RR^d$, usually dense in the group of all translations. We also relax the concept of equidecomposability to ignore sets of Lebesgue measure zero: two sets $A$, $B$ are called $G$-equidecomposable \emph{up to measure zero} if we can remove from them a set of measure zero such that the remaining sets are $G$-equidecomposable. This relaxation is particularly natural if one is to impose the requirement of measurability on the pieces of the equidecomposition. This relaxation does not usually cause any problems in applications of equidecomposability, e.g.\ to tilings \cite{GK26}. Subject to these assumptions and demands, our goal in this paper is generally to achieve equidecomposability with \textit{measurable} pieces.

One can think of the equidecomposability of $A$ and $B$ as a problem of finding a perfect matching in a bipartite graph. Take the bipartite graph with the points of $A$ on one side and the points of $B$ on the other. Then $A, B$ are $G$-equidecomposable if and only if there exists a finite set $F \subset G$ such that the bipartite graph whose edges are all pairs of the form $(a, f \cdot a)$ with $a \in A$, $f \cdot a \in B$, $f \in F$, has a perfect matching. Recall that a \emph{perfect matching} is a collection of disjoint edges that touch all points of $A$ and $B$. Let us call such a perfect matching a \textit{$G$-matching}.

Our main tool in the effort to produce measurable pieces in an equidecomposition is the \textit{measurable Hall's theorem} due to Cie\'{s}la and Sabok \cite{CS22} (see Theorem \ref{thm:cs22} below), which uses an appropriately mixing group action on the ambient space in order to deduce the existence of a \textit{measurable} $G$-matching between two sets $A, B$ from the existence of an arbitrary (not necessarily measurable) $G$-matching. By a measurable $G$-matching we mean a $G$-matching for which the set $A_g = \{a \in A : \text{$(a, g\cdot a)$ is part of the matching}\}$ is measurable for each $g \in G$.

\subsection{}
The structure of the rest of this paper is as follows.

In the preliminary Section \ref{s:hall} we review the equidecomposability concepts that will be used in the paper and formulate the measurable Hall's theorem due to Cie\'{s}la and Sabok \cite{CS22}.

In Section \ref{s:brem-sets} we discuss \textit{bounded remainder sets}, and we show that if two measurable sets $A,B \sbt \R^d$ of the same measure are bounded remainder sets with respect to a totally irrational $d$-dimensional vector $\alpha$, then $A,B$ are equidecomposable with measurable pieces using translations from $\mathbb{Z} \alpha + \mathbb{Z}^d$. 

In Section \ref{s:bde} we prove that if two model sets constructed from the same lattice and two different Riemann measurable windows $W$ and $W'$ are \emph{bounded distance equivalent}, then
the two windows are equi\-decomposable up to measure zero with measurable pieces using translations from $p_2(\Gamma)$, where $\Gamma$ is the common lattice defining the model sets and $p_2(\Gam)$ is its projection onto the subspace containing the windows $W$ and $ W'$. This addresses a gap in the proof of \cite[Theorem 1.1]{Gre25}, see \cite{Gre26}.

The results in Sections \ref{s:brem-sets} and \ref{s:bde} rely on the measurable Hall's theorem \cite{CS22}
which only gives equidecomposability with measurable pieces.
In Section \ref{s:1d}  we use a different method to prove that if two \emph{one-dimensional} model sets are bounded distance equivalent, then the corresponding Riemann measurable windows are equidecomposable \textit{with Riemann measurable pieces} using translations from $p_2(\Gamma)$.
Moreover, if the windows are polytopes
then the pieces of the equidecomposition
can also be chosen to be polytopes;
however this result fails
for model sets in dimensions two and higher.

Finally, the paper includes an appendix where we clarify the status of the results announced in the retracted paper \cite{Gre25}.


\section{Equidecomposability and Hall's condition}\label{s:hall}

In this preliminary section we review
the connection between equidecomposability and Hall's condition,
and state the measurable Hall's theorem due to Cie\'{s}la and Sabok
\cite{CS22} that will be used later on.

\subsection{Equidecomposability}\label{ss:equidecomposition}

Let $X$ be a set endowed with an action of a group $G$.
We use $g \cdot x$ to denote 
the action of an element $g \in G$ on a point $x \in X$.

We say that two sets $A, B \subset X$ are 
\emph{$G$-equidecomposable} 
if there exist finitely many
sets $A_1, \dots, A_n \sbt X$ and elements
$g_1, \dots, g_n \in G$ such that 
$\{A_j\}_{j=1}^{n}$ forms a partition of $A$,
 while $\{g_j \cdot A_j \}_{j=1}^{n}$
forms a partition of $B$. 

We say that  $A, B \subset X$  
satisfy \emph{Hall's condition} with respect to $G$,
if there exists a finite set $F \sbt G$ such that
\begin{enumerate-num}
\item
\label{it:x.hall.a}
$|S | \le |(F \cdot S) \cap B|$
for every finite set
$S \sbt A$;
\item
\label{it:x.hall.b}
$|T| \le |(F^{-1} \cdot T) \cap A|$
for every finite set $T \sbt B$.
\end{enumerate-num}

To motivate this definition,  consider  $A, B$
as two disjoint vertex sets of a bipartite graph,
where two vertices $a \in  A $ and 
$b \in  B $ are connected by an
edge if and only if $b = g \cdot a$ for some
$g \in F$. The conditions \ref{it:x.hall.a}  and
\ref{it:x.hall.b} then say that the size of
every finite set of vertices in $ A  $
or in $ B  $
does not exceed the size of the set of its neighbors in the graph.

The following proposition clarifies the connection between 
the notions of 
equidecomposability and Hall's condition.

\begin{proposition}
\label{prop:x.hall}
Two sets $A, B \subset X$  are
$G$-equidecomposable
if and only if $A$ and $B$
satisfy Hall's condition with respect to $G$.
\end{proposition}

\begin{proof}
We first prove the `if' part.
Suppose  that there is a finite set $F \sbt G$ such that
\ref{it:x.hall.a} and \ref{it:x.hall.b} hold. By the classical 
Hall's marriage theorem, the condition
\ref{it:x.hall.a} implies that for every finite set
$S \sbt A $ there exists an
injective map $\pphi_S: S \to B $
satisfying $\pphi_S(a) \in F \cdot a$ for all $a \in S$.
By an application of Tychonoff's theorem,
see \cite{HV50}, there is an
injective map $\pphi: A  \to B $
such that $\pphi(a) \in F \cdot a$ for all $a \in A $.
In a similar way, we deduce from 
\ref{it:x.hall.b} that there is an
injective map $\psi: B  \to A $
such that $\psi(b) \in F^{-1} \cdot b$ for all $b \in B $.
In turn, the proof of the Cantor-Schr\"{o}der-Bernstein 
theorem (see \cite[Theorem 3.6]{TW16})
yields a bijection $\chi: A  \to B $ such that 
$\chi(a) \in F \cdot a$ for all $a \in A $.
This implies that $A $ and 
$B $ are equidecomposable using
only actions of the finite set $F$.

Next we prove the `only if' part. Suppose that  
$\{A_j\}_{j=1}^{n}$ forms a partition of $A $
and that $\{g_j \cdot A_j \}_{j=1}^{n}$
forms a partition of $B $, where $g_1, \dots, g_n \in G$.
This allows us to define a  bijection $\chi: A \to B$
given  by $\chi(a) = g_j \cdot a$ if $a \in A_j$.
By the necessity part of the classical
Hall's marriage theorem, this implies that 
both conditions
\ref{it:x.hall.a} and \ref{it:x.hall.b} are satisfied
with the finite set $F = \{g_1, \dots, g_n\}$.
\end{proof}

\subsubsection*{Remarks}

1. The proof shows that if $A,B$ 
satisfy Hall's condition with a given finite set $F \sbt G$,
then $A,B$ are
equidecomposable using only actions of the same finite set
$F$, and also the converse it true.

2. In the case where the sets $A,B$ are \emph{countable},
the application of Tychonoff's theorem
can be replaced by a standard diagonalization argument.

\subsection{Equidecomposability up to measure zero}
\label{ss:equid-0}

Let  $(X, \mu)$ be a (finite or infinite) measure space, 
endowed with a  measure preserving action
of a \emph{countable} group $G$.

We say that two measurable sets $A, B \subset X$ are 
$G$-equidecomposable \emph{up to measure zero},
if there exist 
finitely many sets $A_1, \dots, A_n \sbt X$,
elements $g_1, \dots, g_n \in G$
and a full measure subset $X' \sbt X$, such that 
$\{A_j \cap X'\}_{j=1}^{n}$ forms a partition of $A \cap X'$,
 while $\{(g_j \cdot A_j ) \cap X'\}_{j=1}^{n}$
forms a partition of $B \cap X'$. 
If the sets $A_1, \dots, A_n$ can be chosen 
measurable, then we say that $A, B$
are $G$-equidecomposable up to measure zero 
\emph{with measurable pieces}.

Following \cite[Definition 1]{CS22} we say 
that two measurable sets $A, B \subset X$  
satisfy Hall's condition \emph{a.e.}\ with respect to $G$,
if there is a finite set $F \sbt G$ and
a full measure subset $X' \sbt X$, such that 
for every $x \in X'$ we have
\begin{enumerate-num}
\renewcommand{\theenumi}{(\roman{enumi}')}
\renewcommand{\labelenumi}{\theenumi}
\item
\label{it:m.hall.a}
$|S | \le |(F \cdot S) \cap B|$
for every finite set
$S \sbt A \cap (G \cdot x)$;
\item
\label{it:m.hall.b}
$|T| \le |(F^{-1} \cdot T) \cap A|$
for every finite set $T \sbt B \cap (G \cdot x)$.
\end{enumerate-num}

In other words, for almost every $x \in X$ the two
sets $A \cap (G \cdot x)$ and
$B \cap (G \cdot x)$ satisfy Hall's condition
with the same finite set $F \sbt G$.

\begin{proposition}
\label{prop:m.hall}
Let  $(X, \mu)$ be a measure space endowed
with a measure preserving action of a countable group $G$.
Two measurable sets $A, B \subset X$  are
\emph{$G$-equidecomposable} up to measure zero
(with possibly non-measurable pieces)
if and only if $A, B$
satisfy \emph{Hall's condition} a.e.\ with respect to $G$.
\end{proposition}

\begin{proof}
We first prove the `if' part. Assume that 
 there is a finite set $F \sbt G$ and
a full measure subset $X' \sbt X$ such that 
\ref{it:m.hall.a} and \ref{it:m.hall.b} 
hold for every $x \in X'$. Since the group $G$ is countable,
then by replacing $X'$ with 
 $\bigcap_{g \in G} (g \cdot X' )$
we may assume that  $G \cdot X' = X'$,
that is, $X'$ is a $G$-invariant set.
It follows that the two sets
$A' = A \cap X'$ and $B' = B \cap X'$
satisfy Hall's condition with the finite set $F$,
  hence $A', B'$  are $G$-equidecomposable
by \propref{prop:x.hall}. 
As a consequence, $A,B$ are 
$G$-equidecomposable up to measure zero.

To prove the converse `only if' part, suppose now that
$\{A_j \cap X'\}_{j=1}^{n}$ forms a partition of $A \cap X'$
and $\{(g_j \cdot A_j ) \cap X'\}_{j=1}^{n}$
forms a partition of $B \cap X'$,
where $g_1, \dots, g_n \in G$  
and $X'$ is a full measure subset of $X$.
Again by replacing $X'$ with 
 $\bigcap_{g \in G} (g \cdot X' )$
we may assume that  $X'$ is a $G$-invariant set. 
This implies that the two sets  $A \cap X'$ and $B \cap X'$
are $G$-equidecomposable considered
as subsets of the set $X'$.
Hence by
\propref{prop:x.hall}  
 there is a finite set $F \sbt G$ such that 
\ref{it:m.hall.a} and \ref{it:m.hall.b} 
hold for every $x \in X'$. 
\end{proof}

\subsection{The measurable Hall's theorem}\label{ss:meas-hall}
Next we state the measurable Hall's theorem 
proved in \cite{CS22}. 
 The theorem gives conditions guaranteeing that
  two measurable sets $A,B \sbt X$ satisfying
Hall's condition are equidecomposable
 \emph{with measurable pieces}.

Assume now that
 $(X, \mu)$ is a standard Borel probability space, endowed
with a \emph{free} pmp (probability measure preserving) action
of a \emph{finitely generated abelian} group $G$.
We recall that the action of $G$ on $X$ is called \emph{free}
if $g \cdot x \ne x$ for every 
nontrivial element  $g  \in G$ and every $x \in X$.

By the structure theorem for finitely generated abelian groups,
we may assume that
$G  = \Z^d \times \Del$ where $d$ is a nonnegative integer and
$\Del$ is a finite abelian group.

\begin{definition}[{see \cite[Definition 5]{CS22}}]
A measurable set $A \sbt X$ is called \emph{$G$-uniform}
if there exist positive constants $c$ and
$n_0$, such that for almost every $x \in X$  
and for every $n > n_0$ we have
$|A \cap (F_n \cdot x)| \ge c n^d$,
where $F_n := \{0, 1, \dots, n-1\}^d \times \Del$.
\end{definition}

The measurable Hall's theorem 
due to Cie\'{s}la and Sabok
states the following:

\begin{theorem}[{\cite[Theorem 2]{CS22}}]
\label{thm:cs22}
Let $(X, \mu)$ be a standard Borel probability space, endowed
with a free pmp action
of a finitely generated abelian group $G$,
and let 
 $A, B \subset X$ be two measurable $G$-uniform sets.
 Then the following conditions are equivalent:
\begin{enumerate-alph}
\item \label{it:cs.a}
$A$ and $B$
satisfy Hall's condition a.e.\ with respect to $G$;
\item \label{it:cs.b}
$A$ and $B$ are $G$-equidecomposable
 up to measure zero (with 
possibly non-measurable pieces);
\item \label{it:cs.c}
$A$ and $B$ are $G$-equidecomposable 
up to measure zero  with measurable pieces.
\end{enumerate-alph}
\end{theorem}

The equivalence of \ref{it:cs.a} and \ref{it:cs.b}
was given in \propref{prop:m.hall}. 
\thmref{thm:cs22} asserts that these conditions
are also equivalent to \ref{it:cs.c}.
This result will be used below.


\section{Bounded remainder sets}\label{s:brem-sets}

\subsection{}
If $A \sbt \R^d$ is a bounded measurable set, we use  $\1_A$ to denote its indicator function, and we let
\begin{equation}
\label{eq:chi.def}
 \chi_A(x) = \sum_{m \in \Z^d} \1_A(x+m), \quad x \in \R^d,
\end{equation}
be the multiplicity function of the projection of $A$ onto $\T^d = \R^d / \Z^d$.

Let $\alpha = (\alpha_1, \alpha_2, \ldots , \alpha_d)$ be a fixed real vector. We shall assume that the vector $\al$ is \emph{totally irrational}, which by definition means that the numbers $1 , \alpha_1, \alpha_2, \ldots , \alpha_d$ are linearly independent over the rationals.

A bounded measurable set $A \sbt \R^d$ is called a
\emph{bounded remainder set} (BRS) with respect to the vector $\al$ if there is a constant $C= C(A, \al)$ such that
\begin{equation}
\label{eq:brs.def}
\Big| \sum_{k=0}^{n-1} \chi_A(x + k\al) - n \mes A \Big| \leq C
 \quad (n=1,2,3,\dots) 
 \quad \text{a.e.} 
\end{equation}
Here and below, we use $\mes A$ to denote the Lebesgue measure of the set $A$.

Bounded remainder sets form a classical topic in
discrepancy theory, see \cite{GL15} for an overview
of the subject and a survey of basic results.

\subsection{}
It is easy to show that if two bounded measurable  sets $A, B \sbt \R^d$
 are equidecomposable up to measure zero
  using only translations by vectors in $\Z \al + \Z^d$,
and if $A$ is a bounded remainder set, then so is $B$,
see \cite[Proposition 4.1]{GL15}.  

A question posed in \cite[Section 7.2]{GL15} asks whether
a converse statement holds in the following sense:
Let  $A, B \sbt \R^d$ be 
two bounded remainder sets of the same measure.
Is  it true that $A$ and $B$ must be equidecomposable
(up to measure zero, with  measurable pieces)
using translations by vectors in $\Z \al + \Z^d$ only?

It was proved in \cite[Theorem 2]{GL15} that the answer
is affirmative if the sets $A, B$ are assumed to be
Riemann measurable, and moreover,  in this case there
exists  an equidecomposition with Riemann measurable pieces.
It was also proved that if $A,B$ are polytopes,
then the pieces can also be chosen to be polytopes.

However, the question has remained open
in the general case. 
Our goal here is to answer this question affirmatively.

\begin{thm}
\label{thm:brs.equi}
Let $A, B \sbt \R^d$ be 
two  bounded remainder sets with respect to a totally irrational vector $\al$, and suppose that $A$ and $B$ have the same measure. Then $A$ and $B$ are equidecomposable
up to measure zero with measurable pieces,
using translations by vectors in $\Z \al + \Z^d$.
\end{thm}

It follows that equidecomposability provides a method for constructing all
bounded remainder sets.
We also note that, as mentioned  in \cite[Section 7.2]{GL15},
this result allows to extend 
\cite[Theorem 5]{GL15} to all bounded remainder sets.

We now turn to the details of the proof.
In what follows, we assume that the sets $A$ and
$B$ both have positive measure (otherwise we have
nothing to prove).

\subsection{}
 Since $A,B$ are bounded subsets
of $\R^d$,  we can choose a
sufficiently large positive integer $q$ 
and vectors
$r_1, \dots, r_q \in \Z^d$
such that, if we denote
$Q = [0,1)^d$, then the cubes
$Q + r_1, \dots, Q + r_q$ 
are disjoint and their union 
covers both   $A$ and $B$. 
This induces a partition of each set $A$ and $B$
 into subsets 
 $A_i := A \cap (Q + r_i)$ and
 $B_i := B \cap (Q + r_i)$, $1 \le i \le q$.

Let $\Z_q := \Z / q\Z$ be the cyclic group of order
$q$, endowed with its probability Haar measure
assigning the mass $1/q$ to each element.

Now consider the product space
$X = \T^d \times \Z_q$ and denote by
$\mu$ the product probability measure on $X$.
We also consider the finitely generated abelian group 
$G = \Z \times \Z_q$. It induces
a free pmp  action on $X$, where 
the action of the element
$(n,\sig) \in G$ on the point $(x, \tau) \in X$
is given by 
$(n,\sig) \cdot (x, \tau)  = (x + n \al, \sig +\tau )$.

Next, we define two measurable sets $\alta, \altb \sbt X$  by
\begin{equation}
\label{eq:p.1.5}
\alta = \bigcup_{i=1}^{q} A_i \times \{i\}, 
\quad 
\altb = \bigcup_{i=1}^{q} B_i \times \{i\}.
\end{equation}
Here we identify the sets $A_i$ and $B_i$ with their
projections on $\T^d$, which we may do since both
$A_i$ and $B_i$ are contained in the cube $Q + r_i$.

We claim that the 
sets $\alta$ and $\altb$ are $G$-equidecomposable up to measure zero, with possibly non-measurable pieces. It suffices to show that
there is a finite set $F \sbt G$
and a full measure subset $X' \sbt X$, such that
for every point $(x, \tau) \in X'$ there exists a bijection
from $ \alta \cap (G \cdot (x, \tau))$ onto
$ \altb \cap (G \cdot (x, \tau))$
that moves elements using only actions
of the set $F$.

To prove this, we will use a technique similar to
\cite[Section 6.2]{GL18}.

\subsubsection{}
Since $A$ is a  bounded remainder set,
it follows  from \cite[Proposition 2.3]{GL15}
that there is a constant $C$ and 
a full measure subset $\Om \sbt \T^d$
such that
\begin{equation}
\label{eq:brs.x.a.1}
\sup_{n>0} \, \sup_{j\in\Z} \, 
\Big| \sum_{k=j+1}^{j+n} \chi_A(x+k\alpha) - n \mes A \Big| \le C,
\quad x \in \Om.
\end{equation}

The set $X' = \Om \times \Z_q$ is a full measure subset of $X$.
We now fix a point $(x, \tau) \in X'$ 
and consider the set $\alta \cap (G \cdot (x, \tau))$.
We construct an enumeration of
the elements of this set in the following way. Define
\begin{equation}
\label{eq:p.1.6}
A^n = A \cap (x + n \al + \Z^d), \quad n \in \Z,
\end{equation}
and let $\{s_n\}$, $n \in \Z$, be a sequence of integers 
such that 
\begin{equation}
\label{eq:p.1.7}
s_0 = 0, \quad s_{n+1} - s_n = \# A^n
\end{equation}
(we note that each $A^n$ is a finite set, 
and that some of the sets $A^n$ may be empty).
For each $n \in \Z$ we then choose some
enumeration $\{a_j\}$, $s_n \le j < s_{n+1}$, 
of the points in the set $A^n$. We also observe that,
since $A_1, \dots, A_q$ form a partition of $A$,
for each $j$ there is a unique element
 $ \sig_j  \in \{1, \dots, q\}$ such that
$a_j \in A_{\sig_j}$. It is now easy to check that
the sequence $\{(a_j, \sig_j)\}$, $j \in \Z$, 
forms an enumeration of
the set $\alta \cap (G \cdot (x, \tau))$.

We now claim that
\begin{equation}
\label{eq:p.2.3.9}
|s_n  - n \mes A | \le C, \quad n \in \Z.
\end{equation}
Indeed, by \eqref{eq:p.1.6}, \eqref{eq:p.1.7} we have
the equality 
$s_{k+1} - s_k = \chi_A(x + k \al)$. If we sum
this equality over $0 \le k \le n-1$ and use \eqref{eq:brs.x.a.1},
we obtain that \eqref{eq:p.2.3.9} holds for  $n > 0$.
In the  case $n < 0$ we establish 
\eqref{eq:p.2.3.9} similarly, by summing
the equality over $n \le k \le -1$.

\subsubsection{}
In a similar way, we define
\begin{equation}
\label{eq:p.1.9}
B^m = B \cap (x + m \al + \Z^d), \quad m \in \Z,
\end{equation}
and let $\{t_m\}$, $m \in \Z$, be a sequence of integers 
such that 
\begin{equation}
\label{eq:p.2.1}
t_0 = 0, \quad 
t_{m+1} - t_m = \# B^m.
\end{equation}
We  choose an 
enumeration $\{b_j\}$, $t_m \le j < t_{m+1}$, 
of the points in the set $B^m$, and let
 $ \tau_j  \in \{1, \dots, q\}$ be
 the unique element  such that
$b_j \in B_{\tau_j}$. We thus obtain an enumeration
$\{(b_j, \tau_j)\}$, $j \in \Z$, of
the set $\altb \cap (G \cdot (x, \tau))$.

Moreover, since $B$ is a  bounded remainder set,
we may assume that the constant $C$ and the
full measure subset $\Om \sbt \T^d$ have been chosen 
such that we have
\begin{equation}
\label{eq:p.2.3.11}
|t_m -  m \mes B | \le C, \quad m \in \Z.
\end{equation}

\subsubsection{}
We now claim that there exists a finite set $E \sbt \Z$,
which does not depend on the point $(x, \tau)$,
such that 
\begin{equation}
\label{eq:p.2.8}
b_j - a_j \in E \al + \Z^d, \quad j \in \Z. 
\end{equation}

Indeed, given $j$ there exist $n,m$ such that 
$a_j \in A^n$ and $b_j \in B^m$. Hence
\begin{equation}
\label{eq:p.4.1}
b_j - a_j \in  (m-n) \al + \Z^d
\end{equation}
which follows from  \eqref{eq:p.1.6}, \eqref{eq:p.1.9}.
We now write
\begin{equation}
\label{eq:p.2.7}
m - n = \Big(m - \frac{t_m}{\mes B}\Big) 
+ \Big( \frac{t_m}{\mes B} - \frac{s_n}{\mes A} \Big)
 +  \Big( \frac{s_n}{\mes A} - n\Big).
\end{equation}
Due to \eqref{eq:p.2.3.9} and \eqref{eq:p.2.3.11}, 
the first and third terms on the right hand side
are bounded in modulus by a certain constant
$K_1 = K_1(A,B)$. 
To estimate the second term, note that 
$s_n \le j < s_{n+1}$ and 
$s_{n+1} - s_n = \# A^n$
which cannot exceed $q$, hence
$0 \le j - s_n < q$. In a similar way, 
$0 \le j - t_m < q$. As a consequence,
$|t_m - s_n| < q$. Since 
$A$ and $B$ have the same measure,
it then follows that also the second term on the right hand side
of \eqref{eq:p.2.7} is bounded in modulus by some constant
$K_2 = K_2(A,B)$. 
We conclude that $m-n$  lies in some finite set $E \sbt \Z$ 
that does not depend on the point $(x, \tau)$.
Hence, \eqref{eq:p.4.1} implies \eqref{eq:p.2.8}.

\subsubsection{}
We now define  $F := E \times \Z_q$, which
is  a finite subset of $G$.  
It follows from 
 \eqref{eq:p.2.8} that for each $j \in \Z$,
 the two points
$(a_j, \sig_j)$ and  $(b_j, \tau_j)$
of the space $X$ differ by an element of the set 
$E \al  \times \Z_q$.
In other words, this means that
$(b_j, \tau_j) \in F \cdot (a_j, \sig_j)$.
As the sequence
$\{(a_j, \sig_j)\}$ is an enumeration of
 $\alta \cap (G \cdot (x, \tau))$,
while the sequence $\{(b_j, \tau_j)\}$ is an enumeration of
 $\altb \cap (G \cdot (x, \tau))$,
this shows that there exists a bijection
from $ \alta \cap (G \cdot (x, \tau))$ onto
$ \altb \cap (G \cdot (x, \tau))$
that moves elements using only actions of the set $F$.
As this holds for every $(x,\tau) \in X' = \Om \times \Z_q$
 which is a full measure subset of $X$, and since the
 finite set $F$ does not depend on the point
$(x,\tau) $, it follows that
 $\alta, \altb$ are $G$-equidecomposable up to measure zero, with possibly non-measurable pieces.

\subsection{}
We now wish to invoke \thmref{thm:cs22} in order to
conclude that the two sets $\alta$ and $\altb$
are $G$-equidecomposable up to measure zero
with measurable pieces.
To this end, we need to verify that the sets
$\alta$ and $\altb$ are $G$-uniform.

Let $F_n := \{0, 1, \dots, n-1\} \times \Z_q$.
To prove that $\alta$ is  $G$-uniform, we need
to show that there are positive
constants $c$ and
$n_0$, such that for all $(x, \tau)$ in some
full measure subset of $X$ and for every $n > n_0$, we have
\begin{equation}
\label{eq:p.5.1}
|\alta \cap (F_n \cdot (x, \tau))| \ge c n.
\end{equation}
We check that this holds 
for all $(x, \tau) \in X' = \Om \times \Z_q$.
Indeed, observe that the elements of the set
$\alta \cap (F_n \cdot (x, \tau))$ are given in our
enumeration as $\{(a_j, \sig_j)\}$, $s_0 \le j < s_n$,
and therefore this set contains exactly $s_n$ elements.
In turn, it follows from \eqref{eq:p.2.3.9} that
we have $s_n \ge n \mes A - C$.
Hence, we can choose 
 $c>0$ small enough and $n_0$
 large enough, not depending on
the point $(x, \tau)$, 
such that \eqref{eq:p.5.1} holds for every $n > n_0$.
This shows that $\alta$ is a $G$-uniform set.

In a similar way, it can be shown that also the
set $\altb$ is $G$-uniform.

\subsection{}
We can therefore apply \thmref{thm:cs22}
and conclude that the two sets $\alta$ and $\altb$
are $G$-equidecomposable up to measure zero with measurable pieces.
Finally, we need to show that this implies  that 
$A, B \sbt \R^d$ are equidecomposable
up to measure zero with measurable pieces,
using only translations by vectors in $\Z \al + \Z^d$.

First, by refining the pieces in the
equidecomposition if needed,
we may assume that each piece of $\alta$
is entirely contained in one of the sets
$A_i \times \{i\}$, $1 \le i \le q$.
Hence, if $P'$ is one of the pieces of 
$\alta$, then $P' = P \times \{i\}$ for some 
$i \in \{1, \dots, q\}$ and for some measurable set $P
 \sbt A_i = A \cap (Q + r_i)$.
The piece $P'$ is carried by some element 
$(n, \sig) \in G$  onto a piece 
$R'$ of the set $\altb$.
If we choose $j \in \{1, \dots, q\}$ such that
$j = i + \sig \pmod{q}$, then $R' = R \times \{j\}$
for some measurable set $R
 \sbt B_j = B \cap (Q + r_j)$.
The fact that $(n, \sig) \cdot P' = R'$ implies that
$P$ and $R$ are equidecomposable 
using translations by vectors from $n \al + \Z^d$.
It remains to note that as $P'$ goes
through all the pieces of $\alta$, the corresponding
sets $\{P\}$ and $\{R\}$ form partitions
of $A$ and $B$ respectively, up to measure zero.
It thus follows that 
$A$ and $B$ are equidecomposable
up to measure zero with measurable pieces,
using translations by vectors in $\Z \al + \Z^d$.


\section{Bounded distance equivalent cut-and-project sets}
\label{s:bde}

\subsection{}
Two discrete point sets $\Lambda,
\Lambda' \sbt \mathbb{R}^m$ are said to be
 \emph{bounded distance equivalent} with constant $K>0$
 if there exists a bijection $\chi :  \Lambda \to \Lambda'$ 
 satisfying
\begin{equation}
| \chi(\lambda) - \lambda | \le K, \quad \lam \in \Lam.
\end{equation}
We indicate this using 
the shorthand notation $\Lambda \bde \Lambda'$.

Let $\Gam$ be  a lattice in $\mathbb{R}^m \times \mathbb{R}^n$.
Denoting the projections from $\mathbb{R}^m \times \mathbb{R}^n$ onto $\mathbb{R}^m$ and $\mathbb{R}^n$ by $p_1$ and $p_2$ respectively, we assume that $p_1 |_{\Gamma}$ is injective, and that the image $p_2(\Gamma)$ is dense in $\mathbb{R}^n$. If $W \sbt \R^n$ is a bounded set  (called a ``window'') then the set
\begin{equation}
\Lambda (\Gamma, W) = \{ p_1(\gamma)  :  \gamma \in \Gamma , \, p_2(\gamma) \in W \} 
\end{equation}
is called 
the \emph{cut-and-project set}, or the \emph{model set}, in $\mathbb{R}^m$ obtained from the lattice $\Gam$ and the window $W$.

It is well-known that a one-dimensional model set $\Lambda (\Gamma, W) \sbt \R$ constructed from a Riemann measurable window $W$ is bounded distance equivalent to an arithmetic progression if and only if, after an appropriate linear transformation, $W$ is a bounded remainder set with respect to a totally irrational vector associated with the lattice $\Gam$, see \cite{HK16}, \cite{HKK17},
\cite[Section 6]{GL18}, \cite[Theorem 4.5]{FG18}.
This suggests that results on bounded remainder sets may admit natural extensions for model sets. 

For instance, \cite[Theorem 1]{GL15} states that any parallelepiped $W \sbt \R^d$ 
spanned by linearly independent vectors in $\Z\alpha + \Z^d$ is a bounded 
remainder set with respect to the totally irrational vector $\al$. This can be seen as a 
special case of \cite[Theorem 3.1]{DO90} which  states that if the parallelepiped $W$ is 
 spanned by vectors in $p_2(\Gam)$, then the
model set $\Lam(\Gam,W)  $ is  bounded distance equivalent to a lattice.

This analogy prompts the question as to whether Theorem \ref{thm:brs.equi} admits (at least, for Riemann measurable sets) an extension to model sets. An affirmative answer would provide a converse to \cite[Theorem 6.1]{FG18} which states that for any finite partition $W=\cup W_i$, 
and for any vectors $\gamma_i \in \Gamma$ 
such that the pieces $W_i + p_2(\gamma_i)$
are disjoint, the two model sets $\Lambda(\Gamma, W)$ and $\Lambda (\Gamma, \cup (W_i + p_2(\gamma_i)))$ are bounded distance equivalent.

Our next theorem provides such a converse result.

\begin{theorem}
\label{thm:gre25.main}
Let $W, W' \sbt \R^n$ be two bounded Riemann measurable sets
of positive measure.
 If the model sets $\Lambda(\Gamma, W)$ and $\Lambda(\Gamma, W')$
  are bounded distance equivalent, then $W, W'$ are
equidecomposable up to measure zero with measurable pieces,
 using only translations by vectors from $p_2(\Gamma)$. 
\end{theorem}

This result was previously announced in
\cite[Theorem 1.1]{Gre25} but the proof given there
was found to contain a gap, see \cite{Gre26}. 
The remainder of this section is devoted to a new 
proof of \thmref{thm:gre25.main} based on
the measurable Hall's theorem.

We note that in \cite{Gre25} the additional claim was made
that if $W, W'$ are polytopes, then the pieces of the 
equidecomposition can also be chosen to be polytopes.
The new proof given below 
\emph{does not} establish this additional claim.
In fact, we will see in \secref{s:1d} that this claim is valid only for \emph{one-dimensional} model sets. 
As a consequence, several results
announced in \cite{Gre25}  are false; we describe these results in  Appendix \ref{sec:gre25}.

\subsection{}
We now turn to the proof of \thmref{thm:gre25.main}.
Assume that the model sets
$\Lambda(\Gamma, W)$ and $\Lambda(\Gamma, W')$ 
are bounded distance equivalent, i.e.\ there
is a constant $M > 0$ and a bijection
$\pphi: \Lambda(\Gamma, W) \to \Lambda(\Gamma, W')$ 
which moves points by distance at most $M$.
As in the original proof given in \cite[Section 3]{Gre25}, we introduce the ``lifted'' sets 
\begin{equation}
\label{eq:lf.2.1}
\Gamma_{W} = \{ \gamma \in \Gamma  : p_2(\gamma) \in W \}, 
\quad
 \Gamma_{W'} = \{ \gamma \in \Gamma : p_2(\gamma) \in W' \},
\end{equation}
and consider a bijection
$\psi: \Gamma_W \to \Gamma_{W'}$ 
defined by the condition
$p_1(\psi (\gam)) = \pphi(p_1(\gam))$ for 
$\gam \in \Gam_W$. This condition
determines $\psi$ uniquely, since
 $p_1|_\Gam$ is an injective map.
 
We claim that there is a constant $K>0$ such that
$|\psi(\gam) - \gam| \le K$ for all  
$\gam \in \Gamma_W$, namely, $\psi$ is a bounded
distance equivalence map from
$\Gamma_W$ onto $\Gamma_W'$.
Indeed, observe that
$|p_1(\psi(\gam) - \gam)| 
= |\pphi(p_1(\gam)) - p_1(\gam)| \le M$; while
$p_2(\psi(\gam) - \gam) \in W' - W$
which is a bounded set in $\R^n$.
We thus conclude that the two sets
$\Gamma_W$ and $\Gamma_W'$
are bounded distance equivalent.

Let us denote $N  = \{ \gam \in \Gam : p_2(\gam) = 0\}$.
Then $N$ is a sublattice of $\Gam$
(remark that if $p_2|_\Gam$ is injective,
then   $N = \{0\}$).
In turn, there is a sublattice 
$L$  of $ \Gam$ such that we have  the direct
sum decomposition
\begin{equation}
\Gam = L \oplus N
\end{equation}
(see  \cite[I.2.2, Corollary 3]{Cas97}).
Then $p_2|_L$ is injective, and
$p_2(L) = p_2(\Gam)$.
 Define
\begin{equation}
L_W = \{ \gam \in  L : p_2(\gam) \in W\}, 
 \quad
L_{W'} = \{ \gam \in L : p_2(\gam) \in W'\},
\end{equation}
then it follows that
\begin{equation}
\label{eq:l.2.2}
\Gamma_W = L_W \oplus N, \quad \Gamma_{W'} = L_{W'} \oplus N.
\end{equation}

\subsection{}
We wish to prove that $L_W$ and $L_{W'}$
are bounded distance equivalent. 
We will obtain this as a consequence of the following lemma.

\begin{lemma}\label{lm:product-bde}
Let $A, B \subset \Z^r$ and suppose that $A\times\Z^s \bde B\times \Z^s$
with constant $K$. Then also $A \bde B$ with the same constant $K$.
\end{lemma}

\begin{proof}
By assumption  there exists  a  bijection
$\chi:  A\times\Z^s \to B\times\Z^s$
that moves points by distance at most $K$.
We consider $A, B$ 
as subsets of $\Z^r$ viewed as a group acting 
on itself by translations.
To prove the claim it
 suffices to show 
 that $A,B$  are equidecomposable
using only actions of the finite set
 $F = \{j \in \Z^r : |j| \le K\}$.
 In turn, by
 \propref{prop:x.hall} it suffices to check
that  $A,B$  satisfy Hall's condition with the finite set $F$.
That is, we need to show that 
$|S| \le |(S + F) \cap B |$
for any  finite set $S \sbt A$, 
and that $|T| \le |(T - F) \cap A |$
for any  finite set $T \sbt B$.
We will  only check that the first condition holds, 
 as the second condition
can be established similarly.

Let $S \sbt A$ be a finite set. Then
for any positive integer $R$,  the bijection
 $\chi$ maps  the set
$S \times \{0,  \dots, R-1\}^s$ 
injectively into
$((S + F) \cap B ) \times \{-K, \dots, R+K-1\}^s$.
Hence
\begin{equation}
|S| \cdot R^s \le |(S + F) \cap B | \cdot (R + 2K)^s,
\end{equation}
and letting $R \to + \infty$ we conclude that
$|S| \le |(S + F) \cap B |$, as we had to show.
\end{proof}

Since $\Gam_W $ and $\Gam_{W'}$ 
are bounded distance equivalent,  
it follows from \eqref{eq:l.2.2}
that after applying a suitable invertible 
linear transformation,  we may use 
\lemref{lm:product-bde} in order to conclude that
also  $L_W$ and $L_{W'}$
are bounded distance equivalent.

\subsection{}
Let $K$ be the bounded distance equivalence constant of 
$L_W$ and $ L_{W'}$.

\begin{lemma}
\label{lm:l.1.3} 
If $x \in \R^n$ satisfies the condition
\begin{equation}
\label{eq:l.1.2}
(\partial W -x) \cap p_2(\Gamma) 
= (\partial W' -x) \cap p_2(\Gamma) = \emptyset,
\end{equation}
then
$ L_{W-x} \bde L_{W'-x} $
with the same constant $K$.
\end{lemma}

\begin{proof}
Let $F = \{\gam \in L : |\gam| \le  K \}$
which is a finite subset of  $L$.
Since $L_W \bde L_{W'}$ with constant $K$, there is  a  bijection
 $\chi: L_W \to L_{W'}$ 
 and a function $f : L_W \to F$ 
 such that 
 $\chi(\tau) = \tau + f(\tau)$ 
 for all $\tau \in L_W$.
Fix a point $x \in \R^n$ satisfying \eqref{eq:l.1.2}, and consider
the sets $A = L_{W-x}$ and $B = L_{W'-x}$
as subsets of $L$ viewed as a group acting on itself by translations.
It suffices to show 
 that $A,B$  are equidecomposable
using only actions of the finite set $F$.

In turn, by \propref{prop:x.hall} it suffices to check
that  $A,B$  satisfy Hall's condition with the finite set $F$.
We will do this by showing that given a finite set
$S \sbt A$ there is an
injective map $\pphi: S \to B$ 
satisfying $\pphi(\gam) \in \gam + F $ for all $\gam \in S$;
and  given a finite set
$T \sbt B$ there is an
injective map $\psi: T \to A$ 
satisfying $\psi(\gam) \in \gam - F $ for all $\gam \in T$.
We will only prove the first claim, as the second claim
can be proved similarly.

Let $S \sbt A = L_{W -x}$ be a finite set. 
Since the image $p_2(L) = p_2(\Gam)$ is dense in $\R^n$,
we may choose a sequence $\gam_j \in L$ such that
$x_j = p_2(\gamma_j) \to x$.
The assumption that $\partial W -x $ does not 
intersect $p_2(\Gamma)$ implies 
that the elements of the finite set $p_2(S)$ lie in
the interior of $W - x$. Hence,
there is $j_0$ such that 
$p_2(S) \sbt W - x_j$
for all $j > j_0$. This means that
\begin{equation} 
S \sbt L_{W - x_j} = L_W - \gam_j,
\end{equation}
and therefore for each $\gam \in S$ there is $\tau_j(\gam) \in L_W$
such that $\gam = \tau_j(\gam) - \gam_j$.
Since both $S$ and $F$ are finite sets, then by
 passing to a subsequence if needed we may assume that
 for each $\gam \in S$ the value
$f(\tau_j(\gam))$ does not depend on $j$, so there is a 
function $h : S \to F$ such that
$f(\tau_j(\gam)) = h(\gam)$ for every $j$ and every $\gam \in S$.
Define $\pphi(\gam) = \gam + h(\gam)$ for each
$\gam \in S$. It remains to show that $\pphi$
is an  injective map from $S$ into $B$.

We first check that $\pphi$ indeed maps  $S$ into $B$. 
Let $\gam \in S$, then
\begin{equation}
\label{eq:mp.1.1}
\pphi(\gam) = \gam + h(\gam) = \tau_j(\gam) - \gam_j + f(\tau_j(\gam))
= \chi(\tau_j(\gam)) - \gam_j.
\end{equation}
Since $\chi$ maps $L_W$ into $ L_{W'}$ then
\begin{equation}
p_2( \pphi(\gam) ) = 
p_2( \chi(\tau_j(\gam)) ) - x_j \in W' - x_j,
\end{equation}
and letting $j \to \infty$ we obtain
that $p_2(\pphi(\gam)  ) $ lies in the closure of
$ W' - x $. In turn, using
the assumption that $\partial W' - x$ does not 
intersect $p_2(\Gamma)$, we conclude
that  $p_2(\pphi(\gam)  ) $ must in fact lie
in the interior of $W' - x$. As a consequence,
$\pphi(\gam) \in L_{W' - x} = B$.

Lastly, we show that $\pphi$ is injective. Indeed, let
 $\gam, \gam' \in S$, then by  \eqref{eq:mp.1.1} we have
\begin{equation}
\label{eq:mp.1.2}
\pphi(\gam) = \chi(\tau_j(\gam)) - \gam_j, \quad
\pphi(\gam') = \chi(\tau_j(\gam')) - \gam_j.
\end{equation}
Hence, if we assume that
$\pphi(\gam) = \pphi(\gam')$
then $\chi(\tau_j(\gam)) = \chi(\tau_j(\gam')) $.
Since $\chi$ is an injective map, it follows that
$\tau_j(\gam) = \tau_j(\gam') $.
But recalling that 
$\gam = \tau_j(\gam) - \gam_j$
and
$\gam' = \tau_j(\gam') - \gam_j$
this implies that $\gam = \gam'$. 
Hence $\pphi$ is an  injective map,
and the lemma is proved.
\end{proof}

\subsection{}
Since $W$ and $W'$ are bounded sets in $\R^n$,
and since the image $p_2(\Gam)$ is dense in $\R^n$,
 we may choose a system of $n$ linearly independent 
 vectors $v_1, \ldots , v_n \in p_2(\Gamma)$ 
 which are large enough for $W$ and $W'$ to be contained in the 
 parallelepiped
 \begin{equation}
 \Om =  \{ t_1 v_1 + \dots + t_n v_n : t_1, \dots, t_n \in 
 [ -\tfrac1{2},  \tfrac1{2}) \}.
 \end{equation}
Let $H$ be the subgroup of $\R^n$
generated by the vectors $v_1, \dots, v_n$. Then
$H$ is  a lattice in $\R^n$ and a subgroup of $p_2(\Gamma)$,
and  $\Om$ is a 
fundamental domain of $H$ in $\R^n$.

We now consider the quotient space $X = \R^n/H$,  and let
$\mu$ be the  Lebesgue measure on $X$ normalized such that
$\mu(X)=1$.  Then $G = p_2(\Gamma) / H$ is a finitely generated abelian group
which induces a free pmp action on $(X, \mu)$ by translations.
Since $W, W'$ are contained in the 
fundamental domain $\Om$ of $H$, we may also view $W,W'$
as measurable subsets of $X$, and we observe that $W, W'$ are 
$G$-equidecomposable (up to measure zero) considered as subsets of $X$,
if and only if  $W, W'$ are 
$p_2(\Gam$)-equidecomposable (up to measure zero) as subsets of $\R^n$.

We now wish to prove that $W,W'$ (as subsets of $X$)
satisfy Hall's condition a.e.\ with respect to $G$.
It suffices to show that
there is a finite set $F \sbt \Gam$
and a full measure subset $X' \sbt X$, such that
for every point $x \in X'$ there exists a bijection
from $ W \cap (G + x) $ onto
$ W' \cap (G + x)$
that moves elements using only actions
of the set $p_2(F)$.

We choose $F := \{ \gam \in L : |\gam| \le K \}$ where 
$K$ is the bounded distance  equivalence constant of 
$L_W$ and $ L_{W'}$,  and we let
 $X'$ be the set of points $x \in X$ satisfying the condition \eqref{eq:l.1.2}
(note that this condition is invariant under translations by vectors in $H$,
so it may be viewed as a condition on elements of $X$).
Since $W$ and $W'$ are Riemann measurable sets, their 
boundaries  $\partial W$ and $\partial W'$ are both sets of measure zero,
which implies that  $X'$ is a full measure subset of $X$.

Fix $x \in X'$, and denote
$A = W \cap (G + x)$
and $B = W' \cap (G + x)$.
We observe that  the mapping
$\gam \mapsto p_2(\gam) + x \pmod{H}$
defines a bijection $\pphi: L_{W-x} \to A$,
as well as a bijection $\psi: L_{W'-x} \to B$.
We also recall that by
\lemref{lm:l.1.3} there is a bijection
$\chi:  L_{W-x} \to L_{W'-x} $
such that $\chi(\gam) - \gam \in F$ for all $\gam \in L_{W-x}$.
Hence $\psi \circ \chi \circ \pphi^{-1}$ 
defines a bijection from $A$ onto $B$ that moves
points using only actions
of the finite set $p_2(F)$.
We conclude that $W,W'$ 
satisfy Hall's condition a.e.\ with respect to $G$.

\subsection{}
We now wish to invoke \thmref{thm:cs22} in order to
conclude that the two sets $W, W'$ 
are $G$-equidecomposable up to measure zero
with measurable pieces (as subsets of $X$).
To this end, we need to verify that $W, W'$
are $G$-uniform sets.

By the structure theorem for finitely generated abelian groups,
there exists a direct sum decomposition
$G = M \oplus \Delta$ where $M$ is a free abelian group
of rank $d$, and $\Del$ is a finite abelian group.
We observe that since $p_2(\Gam)$ is dense in $\R^n$,
 then $G$ is dense in $X$.
In turn, this implies that also $M$ must be dense in $X$
(see \cite[Section 2.1]{Rud62}).

Let $e_1, \dots, e_d$ be some basis for $M$, and denote
\begin{equation}
F_k = P_k \oplus \Del, \quad
P_k =  \Big\{ \sum_{j=1}^{d} m_j e_j : m_1, \dots, m_d \in \{0, 1, \dots, k-1 \}  \Big\}.
 \end{equation}
  To prove that $W$ is a $G$-uniform set,
 we must show that there exist positive constants
 $c$ and $k_0$ such that for  almost all $x\in X$ 
 and every $k  > k_0$ we have 
\begin{equation}
\label{lb}
| W \cap (F_k + x) | \ge c k^d.
\end{equation}

Since $W$ is a Riemann measurable set of positive
measure, there is $\eps > 0$ such that $W$ contains
some open ball $U$ of radius $2\eps$.
Since $M$ is dense in $X$, there is a positive
 integer $k_0$ such that the set $P_{k_0}$
forms an $\eps$-net in $X$. This implies that also any translate of $P_{k_0}$
is an $\eps$-net in $X$. 
Now observe that for every $x \in X$ and every $k > k_0$,
the set $P_k  + x$ contains at least
$\lfloor k / k_0 \rfloor^d$ disjoint translated copies of $P_{k_0}$,
and each one of these translated copies must intersect 
the ball $U$. It follows that
\begin{equation}
| W \cap (F_k + x) | \ge | U \cap (P_k + x) | \ge 
\lfloor k / k_0 \rfloor^d \ge c k^d,
\end{equation}
which verifies condition \eqref{lb} and shows that
 $W$ is a $G$-uniform set.
In a similar way, one can show that also the set $W'$
is $G$-uniform.

Finally, by an application of \thmref{thm:cs22} we conclude 
that the two sets $W, W'$ are $G$-equidecomposable up to measure zero
 with measurable pieces as subsets of $X$. This implies that
$W, W'$ are $p_2(\Gamma)$-equidecomposable
up to measure zero with measurable pieces
as subsets of $\R^n$,
and completes the proof of \thmref{thm:gre25.main}.


\section{Equidecomposability with Riemann measurable and polytope pieces}
\label{s:1d}

\subsection{}

Notice that in the statement of
 \thmref{thm:gre25.main},
 the sets $W$ and $W'$
are assumed to be Riemann measurable,
yet  the result only guarantees their 
$p_2(\Gam)$-equidecom\-posability  with measurable pieces.
One may therefore ask whether the pieces in the equidecom\-position may be chosen so that they are also Riemann measurable.

A variant of the question,
which is perhaps even more important, 
is the following: if the sets $W$ and $W'$ in \thmref{thm:gre25.main}
are assumed to be polytopes, can the pieces in the
equidecomposition also be chosen to be polytopes?

Note that by
 a ``polytope" in $\R^d$ we mean any finite union 
of $d$-dimensional simplices with disjoint interiors. 
Thus a polytope may be non-convex, or even 
disconnected.

In this section, we first prove that for \emph{one-dimensional} model sets, the answer to both questions above is affirmative. We then show that for model sets in dimensions two and higher, at least the second question admits a \emph{negative} answer.

\subsection{One-dimensional cut-and-project sets}
Let $\Gamma$ be a lattice in $\mathbb{R} \times \mathbb{R}^d$, such that if $p_1$ and $p_2$ denote the projections from  $\mathbb{R} \times \mathbb{R}^d$ onto $\mathbb{R}$ and $\mathbb{R}^d$ respectively,  then $p_1 |_{\Gamma}$ is injective, 
 while $p_2(\Gamma)$ is dense in $\mathbb{R}^d$.
If $W \sbt \R^d$ is a bounded set, then again we consider the model set
in $\R$ defined by
\begin{equation}
\Lambda (\Gamma, W) = \{ p_1(\gamma)  :  \gamma \in \Gamma , \, p_2(\gamma) \in W \}.
\end{equation}

\begin{theorem}
\label{thm:onedim}
Let $W, W' \sbt \R^d$ be two bounded Riemann measurable sets (resp.\ two polytopes). 
If the one-dimensional model sets 
 $\Lambda(\Gamma, W)$ and $\Lambda(\Gamma, W')$ are bounded distance equivalent, 
 then $W, W'$ are  equidecomposable up to measure zero with Riemann measurable pieces
(resp.\ with polytope pieces)
 using only translations from $p_2(\Gamma)$.
\end{theorem}

The proof below does not rely on the measurable Hall's theorem
which only gives equidecomposability with measurable pieces.
It is rather based on the connection of the problem to
bounded remainder sets and the
results obtained in \cite{GL15}, \cite{GL18}.

\subsection{Lattices in general position}
\label{sec:modelsets}

We say that a lattice $\Gamma$ in
$\R \times \R^d$ is  in \emph{general position}
if the restriction of $p_1$ to $\Gamma$ is injective, 
and the image $p_2(\Gam)$ is dense in $\R^d$.

In \cite{GL18} the term ``general position''
was used to indicate that the restrictions of
both $p_1$ and $p_2$ to
$\Gamma$ are injective, and both their
images $p_1(\Gam)$ and $p_2(\Gam)$ are dense 
in $\R$ and $\R^d$ respectively.
These two definitions are in fact equivalent:

\begin{lemma}
If $\Gamma \sbt \R \times \R^d$ is a lattice in 
general position, then also
the restriction of $p_2$ 
to $\Gamma$ is injective, 
and the image $p_1(\Gam)$ is dense in $\R$.
\end{lemma}

\begin{proof}
Let $v_1, \dots, v_{d+1}$ be a basis for the lattice $\Gam$. 
The assumption that 
$p_2(\Gam)$ is dense in $\R^d$ implies that
$p_2(v_1), \dots, p_2(v_d)$ must be linearly independent
vectors in $\R^d$. Hence the vector  $p_2(v_{d+1})$ 
admits a unique expansion
$p_2(v_{d+1})  = \sum_{j=1}^{d} \al_j p_2(v_j)$.
Using again the assumption that 
$p_2(\Gam)$ is dense in $\R^d$ implies that
the numbers $1, \al_1, \dots, \al_d$ are
rationally independent.
As a consequence,
 the restriction of $p_2$ to $\Gamma$ is injective.
 
Since
the restriction of $p_1$ to $\Gamma$ is injective,
the numbers 
$p_1(v_1), \dots, p_1(v_{d+1})$
must be  rationally independent. Hence these numbers
generate a dense subgroup of $\R$. But this subgroup
coincides with
the image $p_1(\Gam)$, so this image 
is dense in $\R$.
\end{proof}

\subsection{Lattices in special form}
Following \cite[Section 4]{GL18} we define the notion of a lattice of special form.

\begin{definition}
\label{def:speciallat}
We say that a lattice
$\Gamma$ in $\R \times \R^d$ is of \emph{special form} if 
\begin{equation}
\label{eq:special}
 \Gamma = \{ (  n + \beta^{\top}(n \alpha + m), n \alpha + m ) :n \in \Z, \, m \in \Z^d  \}
\end{equation}
where $\alpha$, $\beta$ are column vectors in $\R^d$ satisfying the following conditions:
\begin{enumerate-num}
\item \label{it:alpha} The vector $\alpha = (\alpha_1, \alpha_2, \ldots , \alpha_d)^{\top}$ is such that the numbers $1 , \alpha_1, \alpha_2, \ldots , \alpha_d$ are linearly independent over the rationals;
\item \label{it:beta} The vector $\beta = (\beta_1 , \beta_2, \ldots , \beta_d)^{\top}$ is such that the numbers $\beta_1, \beta_2, \ldots , \beta_d, 1+ \beta^{\top} \alpha$ are linearly independent over the rationals.
\end{enumerate-num}
\end{definition}

It is easy to check that
the conditions imposed on the vectors $\alpha$ and $\beta$ are precisely those necessary and sufficient for $\Gamma$ to be in general position.

Let $a$ be a nonzero real scalar and $B$ be a $d \times d$ invertible real matrix. 
We consider a linear and invertible transformation 
$T$ from $\R \times \R^d$ onto itself given by
\begin{equation}
\label{eq:transt}
T  (x, y) = (ax, By) ,  \quad (x,y )  \in \R \times \R^d.
\end{equation}

\begin{lemma}[{see \cite[Lemma 4.3]{GL18}}]
\label{lem:s.f.1}
Assume that $L \subset \R \times \R^d$ is
 a lattice in general position.
Then there exist a
 lattice $\Gamma$ of special form \eqref{eq:special} 
and an invertible linear   transformation $T$ 
of the form \eqref{eq:transt} such that $T(L)=\Gamma$.
\end{lemma}

We argue that by Lemma \ref{lem:s.f.1} it suffices to consider lattices of special form. For suppose Theorem \ref{thm:onedim} holds in this case, and suppose $\Lambda(L, W)$ and $\Lambda(L, W')$ are bounded distance equivalent. Then so are the ``lifted'' sets 
\begin{equation*}
L_{W} = \{ \ell \in L  : p_2(\ell) \in W \}, 
\quad
L_{W'} = \{ \ell \in L : p_2(\ell) \in W' \},
\end{equation*}
and thus also the sets 
\begin{equation*}
T(L_W) = \{(ap_1(\ell), Bp_2(\ell)) : p_2(\ell) \in W\} = \Gamma_{BW}
\end{equation*}
and 
\begin{equation*}
T(L_{W'}) = \{(ap_1(\ell), Bp_2(\ell)) : p_2(\ell) \in W'\} = \Gamma_{BW'}.
\end{equation*}
It follows that the projected sets $p_1(\Gamma_{BW})=\Lambda(\Gamma, BW)$ and $p_1(\Gamma_{BW'})=\Lambda(\Gamma, BW')$ are bounded distance equivalent in $\R$. Since we assume that Theorem \ref{thm:onedim} holds for the lattice $\Gam$ of special form, this implies that the sets $BW$ and $BW'$ are equidecomposable up to measure zero using translations from $p_2(\Gamma) = Bp_2(L)$. It follows that $W$ and $W'$ are equidecomposable up to measure zero using translations from $p_2(L)$. Finally, since $B$ is a linear and invertible map, properties of the pieces in the partition (such as Riemann measurability or them being polytopes) are preserved.

In what follows we will thus assume that $\Gam$ is 
a lattice of the special form \eqref{eq:special}.

\subsection{Point counting function}
If $\Lam$ is  a uniformly discrete set in $\R$, then
we define its point counting function $\nu(\Lam, x ) $ as
\begin{numcases}{\nu(\Lam, x ) = }
\# (\Lam \cap [0,x)), & $x \ge 0$, \label{eq:ct.1} \\[4pt]
 - \; \# (\Lam \cap [x,0)), & $x < 0$. \label{eq:ct.2} 
 \end{numcases}

\begin{lemma}
\label{lem:b.c.4}
If two uniformly discrete sets 
$\Lam, \Lam' \sbt \R$ are
bounded distance equivalent,
then there is a constant $C$ such that
$|\nu(\Lam, x )  - \nu(\Lam', x ) | \le C$
for all $x \in \R$.
\end{lemma}
 
This is obvious and so the proof is omitted.

\subsection{Point counting for cut-and-project sets}
Let $W$ be a bounded set in $\R^d$, and  
\begin{equation}
\Lambda (\Gamma, W) := \{ p_1(\gamma) : \gamma \in \Gamma , \, p_2(\gamma) \in W \} 
\end{equation}
be the model set in $\R$ generated by the lattice $\Gam$ of special form \eqref{eq:special}
 and the window $W$.
It is well-known that $\Lambda(\Gamma, W)$ is a uniformly discrete set.
We recall that
\begin{equation}
\label{eq:chis}
 \chi_W(x) := \sum_{m \in \Z^d} \1_W(x+m), \quad x \in \R^d,
\end{equation}
denotes  the multiplicity function of the projection of $W$ onto
$\T^d = \R^d / \Z^d$.

\begin{lemma}
\label{lem:b.c.5}
The point counting function of  $\Lambda (\Gamma, W)$ satisfies
\begin{equation}
\label{eq:discr}
\nu(\Lambda (\Gamma, W), N)
 =  \sum_{n=0}^{N-1} \chi_W(n \al)  + O(1), \quad N \to + \infty.
\end{equation}
\end{lemma}

\begin{proof}
Indeed, due to the special form \eqref{eq:special}, the   elements $\gam \in \Gam$
may be parametrized by the vectors $(n,m) \in \Z \times \Z^d$ in such a way
that 
\begin{equation}
\label{eq:pt.a.1}
p_1(\gam) = n + \beta^{\top} (n \al + m), \quad
p_2(\gam) = n \al + m.
\end{equation}
Now the point $p_1(\gam)$ belongs to
$\Lambda (\Gamma, W)$ if and only if
$n \al + m \in W$. In this case
\begin{equation}
 p_1(\gam) - n   =  \beta^{\top} (n \al + m) \in \beta^{\top} W,
\end{equation}
and   $\beta^{\top} W$ is a bounded subset of $\R$.
Hence there is  $C = C(\Gam, W)$ such that
\begin{equation}
| p_1(\gam) - n | \le C
\end{equation}
whenever $p_1(\gam)$ is a 
 point in $\Lambda (\Gamma, W)$.
It follows that 
$\nu(\Lambda (\Gamma, W), N)$ differs from 
\begin{equation}
 \# \{ (n,m) \in \Z \times \Z^d : 0 \le n \le N-1, n \al + m \in W\}
\end{equation}
by a bounded magnitude, which is equivalent to \eqref{eq:discr}.
\end{proof}

\begin{lemma}
\label{lem:d.r.4}
Let $W, W'$ be two bounded, Riemann measurable sets in $\R^d$.
If  $\Lambda (\Gamma, W)$ and
$\Lambda (\Gamma, W')$ are 
bounded distance equivalent,
then there is a constant $C$ such that
\begin{equation}
\label{eq:bd.s.3}
\Big| \sum_{n=0}^{N-1} \chi_W(x+n \al) -
\sum_{n=0}^{N-1} \chi_{W'}(x+n \al) \Big| \le C \quad \text{a.e.}
\end{equation}
holds for every $N$.
\end{lemma}

\begin{proof}
Define
$f(x) = \chi_W(x) -  \chi_{W'}(x)$
and $S_N(x) = \sum_{n=0}^{N-1} f(x + n\al)$.
Lemmas \ref{lem:b.c.4} and \ref{lem:b.c.5}
imply the existence of 
a constant $C$ such that
$|S_N(0)| \le C$ for every $N$.
We now use an argument from
\cite[Proposition 2.2]{GL15}. 
The function $S_N$ is $\Z^d$-periodic and we have
$S_N(x + j \al) = S_{N+j}(x) - S_j(x)$, 
 hence $|S_N| \le 2C$ on the set
$\{j \al\}_{j=1}^{\infty}$
which is dense in $\T^d  = \R^d / \Z^d$.
Since $W$ and $ W'$ are Riemann measurable sets,
the function $S_N$ is continuous at almost every point,
so it follows that $|S_N| \le 2C$ a.e.
\end{proof}

\subsection{Conclusion of the proof of \thmref{thm:onedim}}

We can now use the observations made above in order to prove
\thmref{thm:onedim}. Indeed, due to \lemref{lem:s.f.1}
we may assume that $\Gam$ is 
a lattice of the special form \eqref{eq:special}.
By \lemref{lem:d.r.4} there exists a constant $C$ such
that  the estimate
\eqref{eq:bd.s.3} holds for every $N$.
We now invoke
\cite[Theorem 7.1]{GL15} which asserts that the condition
\eqref{eq:bd.s.3} 
is satisfied if and only if $W, W'$ are equidecomposable 
up to measure zero
with Riemann measurable pieces, using only translations by 
vectors in $\mathbb Z\alpha+\mathbb Z^d = p_2(\Gam)$;
and moreover, if $W, W'$  are two polytopes in $\mathbb R^d$ 
then the pieces in the
equidecomposition may be chosen so that they are also polytopes.
This completes the proof of 
\thmref{thm:onedim}.
\qed

\subsection{Cut-and-project sets in dimensions two and higher}

Finally we show that \thmref{thm:onedim}
\emph{does not} remain true (at least, the claim about
equidecomposition with polytope pieces) for cut-and-project sets in dimensions two and higher. 

Let $\Gam$ be a lattice in $\mathbb{R}^m \times \mathbb{R}^n$ such that $p_1 |_{\Gamma}$ is injective and  $p_2(\Gamma)$ is dense in $\mathbb{R}^n$.

\begin{theorem}
\label{thm:twodimhigher}
For any $m \ge 2$ and $n \ge 1$,
there exist a lattice
$\Gam \sbt \R^m \times \R^n$ and two parallelepipeds $W, W' \sbt \R^n$ such that the two model sets $\Lambda(\Gamma, W)$ and $\Lambda(\Gamma, W')$ in $\R^m$
are bounded distance equivalent, but $W, W'$ are not equidecomposable with polytope pieces using translations by vectors from $p_2(\Gamma)$. 
\end{theorem}

\begin{proof}
The proof is based on 
 \cite[Theorem 1.2]{HKW14}. It follows from part (2) of
 this theorem that there exists a lattice
 (in fact, many lattices in a suitable sense)
$\Gam \sbt \R^m \times \R^n$
such that for any parallelepiped $W \sbt \R^n$
with sides parallel to the coordinate axes,
the model set
$\Lambda(\Gamma, W)$ is bounded distance
equivalent to a lattice in $\R^m$.

(We note that in \cite{HKW14} the authors realize a cut-and-project set as the set of return times of an $m$-dimensional linear flow to an $n$-dimensional linear section inside the $(m+n)$-dimensional torus $(\R / \Z)^{n+m}$; we refer the reader to
\cite[Section 2.5]{ASW22} where
the equivalence of this framework and the cut-and-project scheme is clarified.)

It is well-known that if  $W \sbt \R^n$
is a Riemann measurable set, then the model set $\Lam(\Gam,W)$ has a uniform density
$\mes(W) / \det(\Gam)$. It is also known
that any two lattices in $\R^m$ of the same density
are bounded distance equivalent,
see \cite[Theorem 5.2]{DO90},
\cite{DO91}, \cite[Section 3.2]{Kol97}.
As a consequence, for any two parallelepipeds
$W, W' \sbt \R^n$ with sides parallel to the coordinate axes,
and such that $\mes(W) = \mes(W')$, the two model sets
$\Lambda(\Gamma, W)$ and
$\Lambda(\Gamma, W')$ are bounded distance
equivalent to each other
(since each model set is bounded distance equivalent to a lattice, and the two lattices have the same density because the windows $W,W'$ have the same measure).

Hence, to prove our theorem it would
suffice that we construct
two parallelepipeds
$W, W' \sbt \R^n$ with sides parallel to the coordinate axes,
and with $\mes(W) = \mes(W')$, 
but such that $W, W'$ \emph{are not}
equidecomposable with polytope pieces
using  translations by vectors from $p_2(\Gamma)$.
This can be done as follows.

Let $e_1$ be the first element of the
standard basis in $\R^n$,  and let
$H \sbt \R^n$ be the linear subspace spanned
by the other $n-1$ basis elements
$e_2, \dots, e_n$ (note that if
$n=1$ then we have $H = \{0\}$). 
We consider a real-valued 
function $\varphi$ 
defined on the set of all polytopes in 
$\mathbb R^n$, given by
\begin{equation}
\pphi(A) = \sum_{F} \eps(F)  \Vol_{n-1}(F),
\end{equation}
where $F$ goes through all the facets of
$A$ that are contained in $H + p_2(\gam)$
for some $\gam \in \Gam$; $\eps(F) = +1$
or $-1$ depending on whether 
 the outward normal to the facet
$F$ is pointing in the direction of $e_1$
or in the opposite direction; and
$\Vol_{n-1}(F)$ is the $(n-1)$-dimensional
volume of the facet $F$.

We recall that a \emph{facet} of a polytope $A \sbt \R^n$ is a face of dimension $n-1$.

Let us now observe that the function $\pphi$ is an \emph{additive invariant} with respect to the group of translations by  vectors from $p_2(\Gam)$. This means 
 that (i) it is additive, namely, if $A$ and $B$ are two polytopes with disjoint interiors then $\varphi(A \cup B) = \varphi(A) + \varphi(B)$; 
and (ii) it is invariant with respect 
to translations by $p_2(\Gam)$,
 that is, $\varphi(A) = \varphi(A + p_2(\gam))$ whenever $A$ is a polytope and $\gam \in \Gam$.
In fact, this function $\pphi$ is an example obtained from a more 
general construction of additive 
invariants with respect to an arbitrary
subgroup of all the translations of $\R^n$,
see \cite[Sections 5.1--5.2]{GL15}.

It is now obvious that for two polytopes $A, B$ to be equidecomposable with polytope pieces using   translations  from $p_2(\Gam)$,  it is  necessary  
that $\varphi(A) = \varphi(B)$. Hence,
it will suffice that we construct
two parallelepipeds
$W, W' \sbt \R^n$ with sides parallel to the coordinate axes, and with $\mes(W) = \mes(W')$, but
such that $\pphi(W) \ne \pphi(W')$.

To this end, choose any  parallelepiped $W$  with sides parallel to the axes, such that
one of the facets of $W$ is  contained in $H$,
while the parallel facet is not 
contained in $H + p_2(\gam)$
for any $\gam \in \Gam$. This implies
that $\pphi(W) \ne 0$. Next,
we let $W'$ be a translated copy of $W$, 
chosen so that no facet of $W'$
 lies in 
$\bigcup_{\gam \in \Gam} ( H + p_2(\gam))$.
This ensures that $\pphi(W')=0$. 
We thus obtain two
  parallelepipeds  $W,W' \sbt \R^n$  with
$\mes(W) = \mes(W')$, 
but such that $W, W'$ are not
equidecomposable with polytope pieces
using    translations by vectors from $p_2(\Gamma)$. So this completes the proof
of \thmref{thm:twodimhigher}.
\end{proof}


\appendix

\section{On the results announced in \cite{Gre25}}
\label{sec:gre25}

\subsection{}
The purpose of this appendix is to clarify
the status of the results announced in the
paper \cite{Gre25}.
The paper has been retracted due to a gap
in the proof of \cite[Theorem 1.1]{Gre25},
see \cite{Gre26}. 

In the present paper, we gave a new proof
of this result
(\thmref{thm:gre25.main}) based on
the measurable Hall's theorem \cite{CS22}.
We note however that the new proof
\emph{does not} establish the additional
claim that 
if the sets $W, W'$ are polytopes, then
the pieces of the equidecomposition can also be chosen to be polytopes. 
Moreover, while this claim does hold for \emph{one-dimensional} cut-and-project sets (\thmref{thm:onedim}), \emph{it is false} for cut-and-project sets in dimensions two and higher (\thmref{thm:twodimhigher}).

As a consequence, the proofs of
\cite[Corollaries 1.3, 1.4, 1.7]{Gre25}
break down, since they rely on
the existence of an equidecomposition with polytope pieces. Moreover, 
\cite[Theorem 1.2, part (2)]{HKW14} provides
counterexamples showing that these corollaries, as well as
\cite[Conjecture 1.6]{Gre25}, \emph{are false} for model sets in  dimensions two and higher.

We note that the proof of \cite[Theorem 1.5]{Gre25} is correct.

\subsection{}
The results announced in \cite{Gre25} are still valid though for one-dimensional model sets. In what follows, we summarize these results and provide alternative proofs, based on the relation of the problem to bounded remainder sets.

The following
question was discussed in \cite{Gre25}. Let $ \Lambda(\Gamma, W) \sbt \R^m$ be a model set constructed from a lattice $\Gamma \subset \mathbb{R}^m
\times \mathbb{R}^n$ and a window $W \sbt \R^n$.
When is $ \Lambda(\Gamma, W) $ 
 bounded distance equivalent to a lattice?

The following theorem answers the question for certain one-dimensional model sets constructed from a polytope window $W$.

\begin{thm}
\label{thm:m6.5}
Let $ \Lambda(\Gamma, W) \sbt \R$ be a one-dimensional model set. Then,

\textup{(i)} 
 If $W \sbt \R$ is 
 a half-open interval, $W = [a,b)$, then
 $ \Lambda(\Gamma, W)$
is bounded distance equivalent to an
arithmetic progression if and only if
$b-a \in p_2(\Gamma)$;

\textup{(ii)} 
 If $W \sbt \R$ is a union of finitely many
  disjoint half-open intervals,
\begin{equation}
\label{eq:w6.6.3}
W = [a_1,b_1) \cup [a_2, b_2) \cup \cdots 
\cup [a_s, b_s),
\end{equation}
then $ \Lambda(\Gamma, W)$
is bounded distance equivalent to an
arithmetic progression if and only if
 there exists a permutation
$\sigma$ of $\{1, \ldots, s\}$ such that 
$b_{\sigma(j)}-a_j \in p_2(\Gamma)$ for $1 \leq j \leq s$;

\textup{(iii)} 
 If  $W \sbt \R^2$ is a half-open parallelogram of the form
\begin{equation}
\label{eq:w6.7.1}
W = \{u + t_1 v_1 + t_2 v_2 : 0 \le t_1, t_2 < 1 \},
\end{equation}
where $u \in \R^2$ and $v_1, v_2$ are two linearly independent vectors in $\R^2$,
then $ \Lambda(\Gamma, W)$
is bounded distance equivalent to an
arithmetic progression if and only if, after possibly 
exchanging $v_1$ and $v_2$, we have that
$v_1 \in p_2(\Gam)$ and $v_2 + \lam v_1 \in p_2(\Gam)$
for some $\lam \in \R$.
\end{thm}

\begin{proof}
First, note that for any Riemann measurable
window $W \sbt \R^d$, the model set
$ \Lambda(\Gamma, W)$ is known to have density
$\mes (W) / \det (\Gam)$, and so 
if the one-dimensional model set
is bounded distance equivalent to an
arithmetic progression, then this
 arithmetic progression must have the same
 density. 
 Second, due to \lemref{lem:s.f.1}
we may assume that $\Gam$ is 
a lattice of the special form \eqref{eq:special},
which implies that  $\det(\Gam) = 1$ and 
$\Lam(\Gam, W)$ has density $\mes(W)$.
It also implies that
$p_2(\Gam) = \Z \al + \Z^d$,
where $\al$ is the totally irrational vector
  from  \eqref{eq:special}.

Now, to prove the ``only if'' part
 in (i), (ii) and (iii), assume 
that   $ \Lambda(\Gamma, W)$ 
is   bounded distance equivalent to
an arithmetic progression of density $\mes(W)$.
It then follows from \lemref{lem:b.c.4} that the 
counting function $\nu(\Lambda(\Gamma, W), x ) $ 
defined by \eqref{eq:ct.1} and \eqref{eq:ct.2} 
satisfies $\nu(\Lambda(\Gamma, I), x ) =  \mes(W) \cdot x + O(1)$
as $x \to + \infty$.
In turn, using \lemref{lem:b.c.5} and 
\cite[Proposition 2.2]{GL15} this
yields that $W$ is a bounded remainder set
with respect to the totally irrational vector $\al$.
The conclusions  in (i), (ii) and (iii)
then follow from
\cite[Proposition 2.4]{GL15},
\cite[Theorem 5.2]{GL15}
and \cite[Theorem 3]{GL15} respectively.

In order to prove the ``if'' part in (i), (ii) and (iii), we observe that
 the assumptions imply using
\cite[Theorem 2.6]{GL15},
\cite[Theorem 5.2]{GL15}
and \cite[Theorem 3]{GL15} 
that $W$ is a bounded remainder set
with respect to $\al$.
Moreover, it can be inferred from the proofs
of these results that if 
$W$ is assumed to be half-open, then
the estimate
\begin{equation}
\label{eq:brs.y.3}
\sup_{n>0} \, \sup_{j\in\Z} \, 
\Big| \sum_{k=j+1}^{j+n} \chi_W(x+k\alpha) - n \mes W \Big| \le C
\end{equation}
in fact holds for every $x$ (i.e.\ not just a.e.)
where the constant $C = C(W, \al)$  
does not depend on $x$.
Hence, we can apply e.g.\ \cite[Lemma 6.1]{GL18}
which yields that the model set
$ \Lambda(\Gamma, W)$ 
is   bounded distance equivalent to
an arithmetic progression.
\end{proof}

\subsection{}
Another question discussed in \cite{Gre25} asks the following:
When is the model set $\Lam(\Gam, W+t) $ bounded distance
equivalent to  $ \Lam(\Gam, W)$? 

The next theorem answers the question
 for  one-dimensional model sets
 constructed from a half-open interval
window $W = [a,b)$.

\begin{thm}
\label{thm:m6.9}
Let $ \Lambda(\Gamma, W) \sbt \R$ be a one-dimensional model set
constructed from a 
 half-open interval  window $W = [a,b)$.
 Then  we have
 \begin{equation}
 \label{eq:g4.9}
  \Lam(\Gam, W+t) \bde \Lam(\Gam, W)
  \end{equation}
  if and only if   $b-a \in p_2(\Gam)$ or
  $t \in p_2(\Gam)$.
\end{thm}

\begin{proof}
Again  due to \lemref{lem:s.f.1}
we may assume that $\Gam$ is 
a lattice of the special form \eqref{eq:special},
which implies that 
$\Lam(\Gam, W)$ has density $b-a$
and 
$p_2(\Gam) = \Z \al + \Z$.

If $b-a \in p_2(\Gamma)$, then
\cite[Theorem 3.1]{DO90} yields that
both  $ \Lambda(\Gamma, W)$ and
 $ \Lambda(\Gamma, W+t)$ 
are bounded distance equivalent to an
arithmetic progression of density $b-a$,
which implies \eqref{eq:g4.9}.
 If $t \in p_2(\Gam)$,  then  
$t = p_2(\gam)$ for some $\gam \in \Gam$,
 and so in this case  we have
$ \Lambda(\Gamma, W+t) = 
 \Lambda(\Gamma, W) + p_1(\gam)$
 and again  \eqref{eq:g4.9} follows.

In the converse direction, assume that 
\eqref{eq:g4.9} is valid. Then by
\lemref{lem:d.r.4} there exists a constant $C$ such
that  the estimate
\eqref{eq:bd.s.3} holds for every $N$
with $W' = W+t$.
This implies, see
\cite[Section 7.3]{GL15}, that
$b-a \in \Z \alpha + \Z$
 or $t \in \Z \alpha + \Z$.
\end{proof}



\begin{thebibliography}{HKW14}

\bibitem[ASW22]{ASW22}
F. Adiceam, Y. Solomon, B. Weiss,
Cut-and-project quasicrystals, lattices and dense forests.
J. Lond. Math. Soc. (2) \textbf{105} (2022), no. 2, 1167--1199.

\bibitem[Cas97]{Cas97}
J. W. S. Cassels,
An introduction to the geometry of numbers.
Springer-Verlag, 1997.

\bibitem[CS22]{CS22}
T. Cie\'{s}la, M. Sabok,
Measurable Hall's theorem for actions of abelian groups.
J. Eur. Math. Soc. (JEMS) 24 (2022), no. 8, 2751--2773.

\bibitem[DO90]{DO90}
M. Duneau, C. Oguey, 
Displacive transformations and quasicrystalline symmetries.
J. Physique \textbf{51} (1990), no. 1, 5--19.

\bibitem[DO91]{DO91}
M. Duneau, C. Oguey, 
Bounded interpolations between lattices.
J. Phys. A \textbf{24} (1991), no. 2, 461--475.

\bibitem[FG18]{FG18}
D. Frettl\"{o}h, A. Garber,
Pisot substitution sequences, one dimensional cut-and-project 
sets and bounded remainder sets with fractal boundary.
Indag. Math. \textbf{29} (2018), no. 4, 1114--1130.

\bibitem[GMP17]{GMP17}
\L. Grabowski, A. M\'ath\'e, O. Pikhurko,
Measurable circle squaring.
Ann. of Math. (2) \textbf{185} (2017), no. 2, 671--710.

\bibitem[Gre25]{Gre25}
Retracted: S. Grepstad, 
Bounded distance equivalence of cut-and-project sets and equidecomposability.
Int. Math. Res. Not. IMRN 2025, no. 4, Paper No. rnaf031, 15 pp.

\bibitem[Gre26]{Gre26}
Retraction of ``Bounded distance equivalence of cut-and-project sets and equidecomposability''.
Int. Math. Res. Not. IMRN 2026, no. 3, Paper No. rnaf376.

\bibitem[GK26]{GK26}
S. Grepstad, M. N. Kolountzakis,
Bounded common fundamental domains for two lattices.
Adv. Math. \textbf{487} (2026), Paper No. 110776.

\bibitem[GL15]{GL15}
S. Grepstad, N. Lev, 
Sets of bounded discrepancy for multi-dimensional irrational rotation.
Geom. Funct. Anal. \textbf{25} (2015), no. 1, 87--133.

\bibitem[GL18]{GL18}
S. Grepstad, N. Lev, 
Riesz bases, Meyer's quasicrystals, and bounded remainder sets.
Trans. Amer. Math. Soc. \textbf{370} (2018), no. 6, 4273--4298.

\bibitem[HV50]{HV50}
P. R. Halmos, H. E. Vaughan,
The marriage problem.
Amer. J. Math. \textbf{72} (1950), 214--215.

\bibitem[HKK17]{HKK17}
A. Haynes, M. Kelly, H. Koivusalo, 
Constructing bounded remainder sets and cut-and-project sets 
which are bounded distance to lattices, II.
Indag. Math. \textbf{28} (2017), no. 1, 138--144.

\bibitem[HKW14]{HKW14}
A. Haynes, M. Kelly, B. Weiss, 
Equivalence relations on separated nets arising from linear toral flows.
Proc. Lond. Math. Soc. (3) \textbf{109} (2014), no. 5, 1203--1228.

\bibitem[HK16]{HK16}
A. Haynes, H. Koivusalo, 
Constructing bounded remainder sets and cut-and-project sets 
which are bounded distance to lattices.
Israel J. Math. \textbf{212} (2016), no. 1, 189--201.

\bibitem[Kol97]{Kol97}
M. N. Kolountzakis,
Multi-lattice tiles.
Internat. Math. Res. Notices 1997, no. 19, 937--952.

\bibitem[Lac90]{Lac90}
M. Laczkovich,
Equidecomposability and discrepancy; a solution of Tarski's circle-squaring problem.
J. Reine Angew. Math. \textbf{404} (1990), 77--117.

\bibitem[Rud62]{Rud62}
W. Rudin, 
Fourier analysis on groups.
Interscience, 1962.

\bibitem[TW16]{TW16}
G. Tomkowicz, S. Wagon,
The Banach-Tarski paradox, 2nd edition.
Cambridge University Press, 2016.


\end{thebibliography}
\end{document}